\documentclass[12pt]{amsart} 
\usepackage{amsmath}
\allowdisplaybreaks

\newtheorem{theorem}{Theorem}[section]

\theoremstyle{definition}

\newtheorem{remark}[theorem]{Remark}

\def\Er{{\mathbb E}}

\def\Pr{{\mathbb P}}

\def\Rr{{\mathbb R}}

\def\Ac{{\mathcal{A}}}

\def\Fc{{\mathcal{F}}}

\def\Rc{{\mathcal{R}}}

\def\({\left(}     
\def\){\right)}    
\def\[{\left[}     
\def\]{\right]}

\def\as{{\frenchspacing a.s.}~}
\keywords{Independence, Approximation, Transport Theory, Optimal Transport  Plan, Monge-Kantorovitch}

\begin{document}
\title[convergence Results]{Convergence Results for Approximation with independent Variables}
\author{Freddy Delbaen and Chitro Majumdar}
\address{Departement f\"ur Mathematik, ETH Z\"urich, R\"{a}mistrasse
101, 8092 Z\"{u}rich, Switzerland}
\address{Institut f\"ur Mathematik,
Universit\"at Z\"urich, Winterthurerstrasse 190,
8057 Z\"urich, Switzerland}
 \address{RSRL; Jumeirah Beach Residence (JBR), Bahar Tower, P.O.Box 29215, Dubai, United Arab Emirates}
  \email{delbaen@math.ethz.ch, chitro.majumdar@rsquarerisklab.com}
\begin{abstract}
For a square integrable m-dimensional random variable $X$ on a probability space $(\Omega,\Fc,\Pr)$ and a sub sigma algebra $\Ac$, we show that there is a constructive way to represent $X-\Er[X\mid\Ac]$ as the sum of a series of variables that are independent of $\Ac$.
\end{abstract}

\maketitle

\section{Notation and Preliminaries}
We use standard probabilistic notation, $(\Omega,\Fc,\Pr)$ is a probability space  and $\Ac \subset\Fc$ is a sub $\sigma-$algebra. All random variables will be square integrable and all norms will be the $L^2$ norm.  For random variables $X\colon \Omega \rightarrow \Rr^m$, the norm is then defined as $\Vert X\Vert^2=\int |X|^2\Pr[d\omega]$, where for $x\in \Rr^m$, $|x|$ denotes the Euclidean norm of $x$.   In \cite{D4} we showed that $X$ can be approximated by a random variable $Y$ that is independent of $\Ac$.  To be more precise and to avoid trivial terms we suppose that $\Er[X \mid \Ac]=0$. We also suppose that the sigma-algebra $\Fc$ is conditionally atomless with respect to $\Ac$.  This property is equivalent to the existence of a $[0,1]$ uniformly distributed random variable, say $U$,  that is independent of $\Ac$, \cite{D3}.  Under these assumptions there is a random variable $Y$ such that 
$$
\Vert X-Y\Vert=\min\left\{  \Vert X-Z\Vert\mid Z \text{ is independent of }\Ac   \right\}.
$$
We also have $\Er[X]=\Er[Y]=0$, where we use integration in $\Rr^m$. We recall that $\Vert X\Vert^2=\Vert X-Y\Vert^2 +\Vert Y\Vert^2$ and that $\Er[X\mid Y]=Y$. These results were stated in different presentations but are not explicitly mentioned in \cite{D4}. We therefore give a sketch on how to prove them.  Let $f:\Rr^m\rightarrow \Rr^m$ be a Borel measurable function such that $f(Y)$ is in $L^2$. Then the optimality of $Y$ implies that for all $\epsilon\in\Rr$:
$$\Vert X - Y\Vert^2 \le\Vert X-Y-\epsilon f(Y)\Vert^2.$$
Expanding this inequality gives
$$
-2\epsilon \Er\[\(X-Y\)\cdot f(Y)\] +\epsilon^2 \Vert f(Y)\Vert^2 \ge 0.
$$
Since this is true for all $\epsilon$ we must have $\Er\[\(X-Y\)\cdot f(Y)\]=0$. This is the same as $\Er[X-Y\mid Y]=0$. But it also implies that $X-Y$ is orthogonal to $Y$ and hence $\Vert X\Vert^2=\Vert X-Y\Vert^2+\Vert Y\Vert^2$.

If $m=1$, it was shown in \cite{D2} that $\Vert Y\Vert \ge \frac{1}{2}\Vert X\Vert_1$ (we needed the $1-$norm there). Applying this result to the coordinate that has the greatest $L^1-$norm we get that for $m\ge 1$: $\Vert Y\Vert \ge \frac{1}{2{m}}\Vert X\Vert_1$. To see this, let us denote the coordinate functions by $\pi_1,\ldots,\pi_m$. Without loss of generality we may suppose that for all $k$ we have $\Er[|\pi_1(X)|] \ge \Er[|\pi_k(X)|]$. It follows that $\Er[|X|]\le \Er[\sum_k |\pi_k(X)|]\le m \Er[|\pi_1(X)|]$. Let now $Z$ be the best approximation of $\pi_1(X)$. The inequality above gives us that $\Vert Z\Vert \ge \frac{1}{2}\Er[|\pi_1(X)|]\ge \frac{1}{2m}\Er[|X|]$.  Since $\Vert Y\Vert \ge \Vert Z\Vert$ the proof is complete.

This bound was the basis for a convergence result. We inductively  define for a given $X$ with $\Er[X\mid\Ac]=0$, the sequence $X_1=X$, and then we define $Y_n$ to be the best approximation of $X_n$ that is independent of $\Ac$.  $X_{n+1}=X_n-Y_n$.  Clearly we have
$$
X_1=Y_1+Y_2+\ldots+Y_n+X_{n+1}\quad\text{and}\quad \Vert X_1\Vert^2=\sum_{k=1}^{k=n}\Vert Y_k\Vert^2+\Vert X_{n+1}\Vert^2.
$$
This shows that $\sum_k\Vert Y_k\Vert^2<\infty$ and hence $\sum_k\Vert X_k\Vert_1^2<\infty$. This implies $X_k\rightarrow 0$ in $L^1-$norm and  $X_k\rightarrow 0$ almost surely.   Since obviously $\sup_n\Vert X_n\Vert \le \Vert X_1\Vert$, this in turn implies that $\sum_{k\ge 1} Y_k$ converges to $X_1$ for every $L^p$ norm with $p<2$. The aim of this paper is to show that the result holds for $p=2$.  

The next section gives a summary of the construction of $Y$.  We do not repeat the measurability problems since the  details can be found in \cite{D4}.  Section 3 then discusses an inequality for the resulting transport problem. This is then used to prove the main result of the paper. For details on transport problems we refer to \cite{D8}.

\section{The Construction of the Solution}
 The idea is to ``decompose" $X$ along the atoms of $\Ac$.  This is done by constructing a factorisation of $X$ through a product space where the first factor  is $\Omega$ equipped with the sigma-algebra $\Ac$.  Then on each atom of $\Ac$ we will find $Y$ as the solution of a Monge-Kantorovitch problem. 
 
  \begin{enumerate}
 \item Let $\Phi\colon \Omega \rightarrow\Omega\times\Rr^m
 \times [0,1]$ be defined as $\Phi(\omega)=(\omega, X(\omega),U(\omega))$ It is measurable for $\Fc$ and $\Ac \otimes \Rc^m \otimes \Rc_{[0,1]}$ where $\Rc^m$ is the Borel $\sigma-$algebra  on $\Rr^m$ and $\Rc_{[0,1]}$ the Borel $\sigma-$algebra on $[0,1]$. $U$ is a fixed random variable, independent of $\Ac$ and uniformly distributed on $[0,1]$.\\
 \item On $\Omega \times \Rr^m \times [0,1]$ we define $w(\omega,x,t)=\omega$, $\chi(\omega,x,t)=x$ and $\tau(\omega,x,t)=t$. Clearly $w\circ \Phi$ is the identity map on $\Omega$, $\chi \circ \Phi=X$ and $\tau \circ \Phi=U.$\\
 \item The image probability $(\Pr\circ\Phi^{-1})$ of $\Phi$ can now be disintegrated,. There exist kernels:\\
 $\mu_{X,U}:\Omega \times (\Rc^m\otimes \Rc_{[0,1]})\rightarrow[0,1],~\Ac-$measurable\\
 \\
 $\mu_X: \Omega\otimes \Rc^m\rightarrow[0,1],~ \Ac-$measurable, such that\\ \\
$\Pr\circ\Phi^{-1}(A \times B\times C)=\int_{A} \Pr[d\omega] \mu_{X,U}(\omega, B \times C), ~~ A\in \Ac, B\in\Rc^m, C\in \Rc_{[0,1]}$\\
 \\
 $\Pr\circ\Phi^{-1}(A \times B\times [0,1])=\int_{A} \Pr[d\omega] \mu_{X}(\omega,B),  A\in \Ac, B\in \Rc^m$
 \end{enumerate}
Almost surely the measure $\mu_X(\omega,.)$ is the marginal distribution of the probability measure $\mu_{X,U}(\omega,.)$.  The equality $\Er[X\mid\Ac]=0$ simply means that almost surely $\int_{\Rr^m} d\mu_X(\omega) \,x=0$.Because of independence the conditional distribution of $U$ (given $\Ac$) is always the Lebesgue measure $\lambda$ on $[0,1]$. 

For a probability measure $\nu$ on $\Rr^m$, having a second moment and each $\omega\in\Omega$, we look at the square of the Wasserstein$-2$ distance, $W^2(\mu_X(\omega),\nu)$ between $\nu$ and $\mu_X(\omega)$.  It turns out, \cite{D4}, that there is  a probability measure $\nu_0$ where the function 
$$\nu\rightarrow \int_\Omega \Pr[d\omega]\, W^2(\mu_X(\omega),\nu)$$
attains its minimum. The measure $\nu_0$ satisfies $\int_{\Rr^m} d\nu_0\, y=0$. We denote by $\kappa(\omega)$ the optimal measure on $\Rr^m\times\Rr^m$ of the  corresponding Kantorovitch problem with penalty function $\frac{1}{2}|x-y|^2$.

Using  results from transport theory we will construct $Y$ for each $\omega\in\Omega$, or better for a set of full measure. $Y(\omega)$ will be a function of $X(\omega)$ and $U(\omega)$. This will guarantee measurability. For almost every $\omega$, $Y$ will have conditional to $\Ac$  the distribution $\nu_0$.  This will guarantee independence and it will in fact conclude the construction of the looked for variable $Y$. The variable $Y$ can be seen as the combination $\eta\circ\Phi$ where $\eta(\omega,x,t)$ is $\Ac \otimes \Rc^m \otimes \Rc_{[0,1]}$ measurable. For each $\omega$ we have $\int d\mu_X(\omega)\,x$=0 (because $\Er[X\mid \Ac]=0$) and $\int d\mu_{X,U}(\omega)\eta(\omega),x,t)=\int_{\Rr^m} d\nu_0\, y=0$. Also by construction $$\int_{\Rr^m\times\Rr^m} |x-y|^2 d\kappa(\omega)=W^2(\mu_X(\omega),\nu).$$
 
\section{An Inequality for the Transport Problem}

We simplify the notation used in the previous section and we will look at two measures $\mu,\nu_0$ on $\Rr^m$. We suppose that $\int_{\Rr^m}x\, d\mu=\int_{\Rr^m} y\,d\nu_0=0$ and the two measures have finite second moments. The Wasserstein distance is obtained by minimising $\int_{\Rr^m\times\Rr^m}\beta(dx,dy)\,|x-y|^2$ over all probability measures with marginals $\mu$ and $\nu_0$. This is the same as maximising the integral
$\int_{\Rr^m\times\Rr^m}\beta(dx,dy)\,x\cdot y$ over the same set of measures $\beta$. Let us denote by $\kappa$ such a maximiser. 
\begin{theorem} An optimal measure $\kappa$ satisfies 
$$
(x\cdot y)^- \le \int_{\Rr^m\times\Rr^m}\kappa(dx,dy) \,(x\cdot y)
$$
implying that $(x\cdot y)^-$ is bounded and
$$                                                                                                                                                                                                                                                                          
\int_{\Rr^m\times\Rr^m}\kappa(dx,dy)\, (x\cdot y)^- \le \int_{\Rr^m\times\Rr^m}\kappa(dx,dy)\, (x\cdot y),
$$
where $(x\cdot y)^-$ is the negative part of the scalar product i.e. $(x\cdot y)^-=-\min(0,(x\cdot y))$.
In case $m=1$ we also get  the  bound $$\frac{1}{4}\(\int  \mu(dx) |x|\)\(\int \nu_0(dy)\,|y|\).$$
\end{theorem}
\begin{remark} Because $\kappa$ yields a bigger integral than $\mu\otimes\nu_0$, we obviously have   $\int_{\Rr^m\times\Rr^m}\kappa(dx,dy) (x\cdot y)\ge \(\int_{\Rr^m}d\mu\,x\)\cdot\(\int_{\Rr^m}d\nu_0\,y\)=0$,
\end{remark}
\begin{proof} The proof is based on the Knott-Smith optimality criterion, \cite{D8}. This theorem says that there is a lower semi continuous convex function $\phi\colon \Rr^m\rightarrow \overline{\Rr}$ such that the support of $\kappa$ is contained in the graph of the subgradient of $\phi$. Because of convexity, the graph is a monotone set in $\Rr^m\times\Rr^m$ --- it is even cyclically monotone but we only need the weaker form of it.  In other words for all choices in the support of $\kappa$, say $(x_1,y_1)$ and $(x_2,y_2)$,  we have $(x_1-x_2)\cdot(y_1-y_2)\ge 0$. This can be rewritten as $x_1\cdot y_1 + x_2\cdot y_2\ge x_1\cdot y_2+x_2\cdot y_1$. If we integrate over $(x_2,y_2)$ with respect to $\kappa$ we get that 
$$
x_1\cdot y_1 \ge - \int \kappa(dx,dy)\, x\cdot y,
$$
for all elements $(x_1,y_1)$ in the support of $\kappa$. Hence $(x\cdot y)^-\le \int \kappa(dx,dy)\, x\cdot y$ ($\kappa-$\as) and integrating again with respect to the measure $\kappa$ yields the desired result.

The case $m=1$ is easier. We know that in this case the two functions $x$ and $y$ are commonotonic and this implies that also $-x^-$ and $y^+$ are commonotonic, similarly we get $x^+$ and $-y^-$ are commonotonic, Because $xy \le 0$ if and only if one of the two alternatives $x\le0,y\ge 0$ or $x\ge 0,y\le 0$ occurs, we get
$\{xy<0\}\subset \(\{x<0\}\cap\{y>0\}\)\cup\( \{x>0\}\cap\{y<0\}\)$. Using commonotonicity we can see that one of the two sets must be empty. Let us suppose that  $\(\{x<0\}\cap\{y>0\}\)$ is nonempty (the other case is similar).
$$
\int_{xy<0} \kappa(dx,dy) \, xy = \int_{x<0,y\ge 0} \kappa(dx,dy)\, xy.
$$
Because of  commonotonicity  we then have
$$
\int_{x<0,y\ge 0} \kappa(dx,dy)\, xy \ge   \(-\int \mu(dx) \,x^-\)\(\int \nu_0(dy) \,y^+\).
$$
Because $\int \mu(dx)\, x=\int \nu_0(dy)\, y=0$ we obtain
$$
\int_{xy<0} \kappa(dx,dy)\, xy \ge- \frac{1}{4} \(\int \mu(dx)\, |x|\)\(\int\nu_0(dy)\,|y|\).
$$
\end{proof}
\begin{remark}   We do not know whether the inequality for $m=1$ can be generalised to the case $m>1$.
\end{remark}
\section{Proof of the Main Theorem}

Before getting into the convergence problem, let us first look at the case where $\Er[X\mid \Ac]=0$, $Y$ is the best approximation to $X$ that is independent of $\Ac$. We use the inequality of the previous section to get a bound on $X-Y$.
\begin{theorem} $$\Er\[\(X\cdot Y\)^-\] \le \Er[X\cdot Y]=\Vert Y\Vert^2.$$
\end{theorem}
\begin{proof} We freely use the notation introduced in section 2. The construction of $Y$ is based on the solution of transport problems, parametrised by $\omega\in\Omega$. For each $\omega\in\Omega$ we get the following inequality -- obtained in section 3.
$$
\int_{\Rr^m\times\Rr^m}d\kappa\, \(\chi\cdot \eta\)^-\le \int_{\Rr^m\times\Rr^m}d\kappa\, \(\chi\cdot \eta\).
$$
Integrating over $\Omega$ we get
$$
\int_\Omega d\Pr\int_{\Rr^m\times\Rr^m}d\kappa\, \(\chi\cdot \eta\)^-=\Er\[\(X\cdot Y\)^-\]
$$
and
$$
\int_\Omega d\Pr\int_{\Rr^m\times\Rr^m}d\kappa\, \(\chi\cdot \eta\)=\Er\[\(X\cdot Y\)\].
$$
Hence $\Er\[\(X\cdot Y\)^-\] \le \Er[X\cdot Y]$ which using $\Er[X\mid Y]=Y$, yields the statement of the theorem.
\end{proof}
Now we put -- as before -- $X_1 = X$ , $Y_n$ the best approximation of $X_n$ ($Y_n$ independent of $\Ac$), $X_{n+1}=X_n-Y_n$. 
\begin{theorem} The sequence $X_n$ tends to zero in $L^2$.  Consequently $X_1=\sum_{k\ge 1} Y_k$ where the series converges in $L^2$.  Moreover $\Vert X_1\Vert^2=\sum_k\Vert Y_k\Vert^2$.
\end{theorem}
\begin{proof} We get the following sequence of inequalities:
\begin{eqnarray*}
|X_2|^2=|X_1-Y_1|^2&=&|X_1|^2+|Y_1|^2-2X_1\cdot Y_1\le |X_1|^2+|Y_1|^2 + 2\(X_1\cdot Y_1\)^-\\
|X_3|^2=|X_2-Y_2|^2&=&|X_2|^2+|Y_2|^2-2X_2\cdot Y_2\\
&\le& |X_2|^2+|Y_2|^2 + 2\(X_2\cdot Y_2\)^-\\
&\le& |X_1|^2+|Y_1|^2 + 2\(X_1\cdot Y_1\)^-+|Y_2|^2 + 2\(X_2\cdot Y_2\)^-\\
\text{recursively we get}&{ }&\\
|X_{n+1|}^2&\le& |X_1|^2 +\sum_{k=1}^{k=n} |Y_k|^2 + 2\, \sum_{k=1}^{k=n} \(X_k\cdot Y_k\)^-.
\end{eqnarray*}
All these variables are bounded by
$$
|X_1|^2 +\sum_{k\ge 1} |Y_k|^2 + 2\, \sum_{k\ge 1} \(X_k\cdot Y_k\)^-,
$$
which is integrable since its expectation is smaller than $\Vert X_1\Vert^2 +3\,\sum_{k\ge 1}\Vert Y_k\Vert^2 <\infty$.
We already have seen that $X_n\rightarrow 0$ in probability (even almost surely) and the dominated convergence theorem now tells us that $X_n\rightarrow 0$ for the $L^2$ norm as well.
\end{proof}


\begin{thebibliography}{99}


\bibitem{D2} F. Delbaen: A bipolar theorem only using  independent random variables, {\em Pure and Applied Functional Analysis}, 2024, forthcoming.



\bibitem{D3}  F. Delbaen: Conditionally Atomless Extensions of Sigma Algebras, {\em arxiv 2003.09254, (2020) }

\bibitem{D4} F. Delbaen and C. Majumdar: Approximation with independent Variables, {\em Front. in Math. Fin., 2 (2), 2023, 141--149}

\bibitem{D8} M. Thorpe: {\em Introduction to Optimal Transport}, {F2.08, Centre for Mathematical Sciences, University of Cambridge, 2018}



\end{thebibliography}
\end{document}